\newif\ifcdc   \cdcfalse
\newif\ifarxiv \arxivfalse
\def\version#1{\csname #1true\endcsname}

\version{arxiv}

\ifcdc\ifarxiv
  \errmessage{Both versions selected: comment out one of \string\version{cdc} /
              \string\version{arxiv}}%
\fi\fi
\ifcdc\else\ifarxiv\else
  \errmessage{No version selected: uncomment one of \string\version{cdc} /
              \string\version{arxiv}}%
\fi\fi

\ifcdc   \input{preamble_cdc.tex}   \fi
\ifarxiv 

\documentclass[12pt]{article}

\usepackage[margin=1in]{geometry}

\usepackage{bm}

\usepackage{amsmath}   
\usepackage{amsfonts}  
\usepackage{amssymb}   
\usepackage{amsthm}    

\usepackage{graphicx}

\newtheorem{theorem}{Theorem}

\newtheorem{lemma}{Lemma}
\newtheorem{corollary}{Corollary}

\usepackage{enumitem}

\newtheorem{definition}{Definition}

\theoremstyle{remark}
\newtheorem{remark}{Remark}

\usepackage{makecell}

\input{mystyle.sty}

\newcommand{\VEC}{\text{Vec}}

\usepackage[colorlinks=true,
            linkcolor=blue,
            urlcolor=blue,
            citecolor=blue,
            anchorcolor=blue]{hyperref}

\title{Inverse Problems for Costs and Controls in LQG MFGs via Mean Field Trajectories}

\author{
  Gr\'egoire Lambrecht\thanks{Center for Data Science, New York University and NYU Shanghai. Partially supported by NSF Award 1922658. \texttt{gl3048@nyu.edu}}
  \and
  Mathieu Lauri\`ere\thanks{Shanghai Center for Data Science; NYU-ECNU Institute of Mathematical Sciences at NYU Shanghai; NYU Shanghai, Shanghai, People's Republic of China. \texttt{mathieu.lauriere@nyu.edu}}
}

\date{}
 \fi

\usepackage{comment}
\ifarxiv
  \includecomment{arxivonly}\excludecomment{cdconly}
  \newcommand{\arxiv}[1]{#1}\newcommand{\cdc}[1]{}
\else
  \excludecomment{arxivonly}\includecomment{cdconly}
  \newcommand{\arxiv}[1]{}\newcommand{\cdc}[1]{#1}
\fi

\ifarxiv
  \newcommand{\ricattiA}{Theorem~\ref{theorem:riccati_1}}
  \newcommand{\ricattiB}{Theorem~\ref{theorem:riccati_2}}
\else
  \newcommand{\ricattiA}{\cite[Thm.~4.1.6, p.~186]{AbouKandilFreilingIonescuJank2003}}
  \newcommand{\ricattiB}{\cite[Thm.~3.1.3, p.~96]{AbouKandilFreilingIonescuJank2003}}
\fi

\begin{document}

\maketitle
\ifcdc
  \thispagestyle{empty}
  \pagestyle{plain}
\fi


\begin{abstract}
	This paper investigates inverse problems for Linear-Quadratic-Gaussian (LQG) Mean Field Games (MFGs) based entirely on the observation of mean-covariance trajectories. We address three sequential challenges: identifying the optimal control for observed initializations, determining the control for arbitrary initializations, and recovering consistent cost parameters. After establishing the existence and uniqueness of the forward Nash equilibrium under mild hypotheses, we analyze the injectivity of the parameter-to-trajectory mapping, demonstrating that it is inherently non-injective and providing sufficient conditions for parameter equivalence. We prove that while the optimal control is locally identifiable for observed initializations under a rank condition, global identifiability requires a deeper structural recovery of the game's costs. To bridge this gap, we propose a constructive semidefinite programming method to infer cost parameters that are 
    consistent with the observed population dynamics. Numerical experiments illustrate this method.
\end{abstract}

\section{Introduction}

Mean Field Games (MFGs) provide a mathematical framework for analyzing strategic interactions in large populations of agents. Since their introduction by Lasry and Lions \cite{LasryLions2007} and Huang, Malhamé, and Caines \cite{HuangMalhameCaines2006}, they have been extensively studied from both partial differential equation and probabilistic perspectives, with a primary focus on the existence, uniqueness, and numerical computation of Nash equilibria \cite{BensoussanFrehseYam2013, CarmonaDelarue2018, AchdouCardaliaguet2020,AchdouCapuzzoDolcetta2010, CarmonaLauriere2021}.

In many applications, the parameters governing agent behavior are not directly observable. Instead, one typically has access to data in the form of state trajectories or aggregate statistics of the population. This naturally leads to \textit{inverse problems in MFGs}, where the objective is to recover underlying model parameters or optimal strategies from observed equilibria. Such problems arise in areas such as economics, engineering, and data-driven control, yet remain relatively underexplored.

Existing works on inverse MFG problems often rely on strong structural assumptions, such as direct access to the value function \cite{liu2023inverse,LiuLoZhang2025}. These assumptions may be restrictive in practical settings where only partial information, such as the evolution of the population distribution, is available~\cite{ZhangYangMouZhou2025}. This motivates the study of corresponding inverse problems relying primarily on observable quantities.

Closely related studies exist regarding the inverse problem in classical control theory \cite{JeanMaslovskayaCDC2018}, as well as its discrete-time stochastic equivalent in the context of maximum entropy inverse reinforcement learning \cite{ChenZiebartAISTATS2015}. Despite these developments in single-agent or finite-player settings, the literature addressing the inverse problem for LQ and LQG games remains rather scarce \cite{Corrigan1973,chen2024inverse,inga2019solution}.

In this paper, we consider inverse problems in the class of \textit{Linear-Quadratic-Gaussian (LQG) Mean Field Games}, where the population distribution is fully characterized by its mean and covariance. This class of games has been extensively studied and has found various applications, see e.g.~\cite{HuangMalhameCaines2006,bardi2012explicit,BensoussanSungYamYung2016,Graber2016}. The LQG structure allows us to formulate the inverse problem in terms of \textit{mean–covariance trajectories}, which are natural and observable (or approximately observable) quantities in many applications.

We address three fundamental questions:
(i) whether the optimal control is locally identifiable from observed initial conditions,
(ii) whether the optimal control is globally identifiable for arbitrary initializations, and
(iii) whether one can infer cost parameters that are 
consistent with the observed equilibria.

\arxiv{Our contributions are as follows. First, we establish existence and uniqueness of Nash equilibria under mild conditions in our formulation. Second, we analyze the \textit{injectivity of the mapping} from model parameters to equilibrium trajectories, showing that it is, in general, inherently non-injective. Third, we provide conditions under which the optimal control is locally identifiable from observed trajectories, and we characterize the information required for global identifiability. Finally, we propose a constructive method to recover \textit{consistent cost parameters} by combining identifiable quantities with a semidefinite programming formulation.

The remainder of the paper is organized as follows. Section \ref{section:background_pb} introduces the model and formalizes the inverse problems. Section \ref{section:existence-inj} studies existence, uniqueness, and injectivity properties. Section \ref{section:main_results} presents the main identifiability results and reconstruction methods. Section \ref{section:numerical} provides numerical experiments, and Section \ref{section:conclusion} concludes.}%
\cdc{Our contributions are as follows. After formalizing the model and the inverse problems (Section~\ref{section:background_pb}), we establish existence and uniqueness of Nash equilibria and analyze the \textit{injectivity of the mapping} from model parameters to equilibrium trajectories, showing that it is inherently non-injective (Section~\ref{section:existence-inj}). We then provide conditions under which the optimal control is locally identifiable, characterize the information required for global identifiability, and propose a constructive method to recover \textit{consistent cost parameters} by combining identifiable quantities with a semidefinite programming formulation (Section~\ref{section:main_results}); Section~\ref{section:conclusion} concludes. For brevity, several proofs are omitted or reduced to a sketch; the complete arguments, together with numerical experiments illustrating the method and an appendix recalling the Riccati results we use, are provided in the long version of this paper \cite{arxiv-version}.}

\section{Background and problem setting}
\label{section:background_pb}

\textbf{Notation.} Let $\NN = \{0,1,\dots\}$ denote the set of natural numbers and $\NN^* = \NN \setminus \{0\}$. For any $k \in \NN$, we define $[k] = \{0,\dots,k\}$ and $[k]^* = [k] \setminus \{0\}$. Let $\RR$ denote the set of real numbers, $\RR^* = \RR \setminus \{0\}$, $\RR_+ = \{x \in \RR \,:\, x \ge 0\}$, and $\RR^*_+ = \{x \in \RR \,:\, x > 0\}$.
For any $n \in \NN^*$, $\RR^n$ denotes the set of real-valued vectors of dimension $n$, and $\RR^{n\times n}$ the set of real-valued square matrices of size $n$. For any matrix $A \in \RR^{n\times n}$, $A^\dag$ denotes the transpose of $A$, $\det(A)$ its determinant, and if $\det(A)\neq 0$, $A^{-1}$ denotes its inverse. $\VEC(A) \in \RR^{n^2}$ denotes the vector formed by stacking the columns of $A$ into a single column vector.
We denote by $\bbS^n$ the set of symmetric matrices, $\bbS^n_+$ the set of symmetric positive semi-definite matrices, $\bbS^n_{++}$ the set of symmetric positive definite matrices, and $\bbS^n_-$ the set of symmetric negative semi-definite matrices. For any matrix $A, B \in \RR^{n\times n}$, we write $A\succeq B$ if and only if for any $x\in \RR^n$, $x^\dag A x \ge x^\dag B x$.
For any $n \in \NN$, we define the commutation matrix $\mathcal{T}_n \in \mathbb{R}^{n^2 \times n^2}$ as the unique linear operator such that $\mathcal{T}_n \VEC(A) = \VEC(A^\dag)$ for all $A \in \mathbb{R}^{n \times n}$.

\subsection{Dynamics and equilibrium}
\label{subsection:dynamics-and-equ}
Let $n \in \NN^*$. For simplicity, we assume the state and action spaces are both $\RR^n$, although one could consider different dimensions. We first define the space of parameters characterizing the agents' environment.

\begin{definition}
	The space of \textbf{admissible dynamics parameters} is defined by $\Gamma_1 = \{(A, B, \sigma) \in (\RR^{n\times n})^3: \det(B) \ne 0 \}$. The space of \textbf{admissible cost parameters} is defined by $\Gamma_2 = \{ (R, Q, Q_T, \bar Q, \bar Q_T) \in \bbS^n_{++}\times (\bbS^n_{+})^4\}$. The space of \textbf{admissible parameters} is then given by the product space $\Gamma = \Gamma_1 \times \Gamma_2$. A generic element of $\Gamma$ (resp. $\Gamma_1, \Gamma_2$) is denoted by $\gamma$ (resp. $\gamma_1, \gamma_2$).
\end{definition}

Let $\gamma \in \Gamma$ represent the parameters of the environment. For a fixed time horizon $T \in \RR^*_+$, let $X_t$ and $\alpha_t$ denote the state and control of a representative agent, respectively.
Let $\bar{x}_0 \in \RR^n$ and $\Sigma_0 \in \bbS_+^n$ be an initial mean and covariance matrix. The agent's state follows the stochastic differential equation:
\begin{equation}
	\label{eq:sde}
	dX_t = (A X_t + B\alpha_t)dt + \sigma dW_t, \quad X_0 \sim \cN(\bar{x}_0, \Sigma_0).
\end{equation}
The agent seeks to minimize the cost functional:
\begin{equation}
	\label{eq:cost-orig}
	\begin{split}
		\cJ(\alpha) = \EE \bigg[ &\int_0^T (X_t^\dag Q X_t + \alpha^\dag_t R \alpha_t + X_t^\dag \bar{Q} \bar{x}_t) dt \\
			&+ X^\dag_T Q_T X_T + X_T^\dag \bar{Q}_T \bar{x}_T \bigg].
	\end{split}
\end{equation}
Here, $\bar x_t$ denotes the mean state of a population of agents following the same control, evolving according to \eqref{eq:sde}. We say that a \textbf{Nash equilibrium} is reached when the agent distribution induced by the optimal control minimizing \eqref{eq:cost-orig} is equal to the population distribution. It is established in the literature that the population distribution satisfies a system of coupled partial differential equations, and under our setting, this distribution remains Gaussian. Consequently, the mean and covariance trajectory satisfies a deterministic system of differential equations. For brevity, we define a Nash equilibrium operationally through the forward-backward equation system governing the mean-covariance trajectory. We refer to \cite{CarmonaDelarue2018, BensoussanFrehseYam2013} for the broader PDE and game-theoretic context.
\begin{definition}
	Let $\gamma \in \Gamma$, and define $2\nu = \sigma \sigma^\dag$. We say that the mean-covariance trajectory $(\bar x, \Sigma)$ is a Nash equilibrium with respect to $\gamma$ if there exists a tuple $(P, r)$ such that $(P, \Sigma, \bar x, r)$ satisfies:
	\begin{align}
		\dot{P}      & + PA + A^\dag P - P B R^{-1} B^\dag P + Q = 0, \label{eq:p_dyn}                                   \\
		\dot{\Sigma} & = (A - B R^{-1} B^\dag P)\Sigma + \Sigma(A - B R^{-1} B^\dag P)^\dag + 2\nu, \label{eq:sigma_dyn} \\
		\dot{\bar x} & = (A  - B R^{-1} B^\dag P)\bar x - B R^{-1} B^\dag r, \label{eq:barx_dyn}                         \\
		-\dot{r}     & = (A^\dag - P B R^{-1} B^\dag)r + \frac{1}{2}\bar{Q}\bar x, \label{eq:r_dyn}
	\end{align}
	subject to the boundary conditions $P_T = Q_T$, $\bar x_0 = \bar{x}_0$, $\Sigma_{0} = \Sigma_0$, and $r_T = \frac{1}{2}\bar{Q}_T \bar{x}_T$.
\end{definition}

In Theorem~\ref{theorem:existence_uniqueness_mfg_lqg} we prove that for any $\gamma \in \Gamma$, there exists a unique tuple $(P, \Sigma, \bar x, r)$ satisfying \eqref{eq:p_dyn}-\eqref{eq:r_dyn}.
The optimal control is given by the feedback policy $\alpha_t = \pi(t, X_t)$, where $\pi(t, x) = -R^{-1}B^\dag ( P_t x + r_t)$ for each $(t, x) \in [0, T] \times \RR^n$, which justifies the next definition.
\begin{definition}
	For any admissible tuple of parameters $\gamma \in \Gamma$ let us define the \textbf{trajectory mapping}
	\begin{align*}
		\psi_\gamma: \RR^n \times \bbS^{n}_{++} & \to C^1([0, T], \RR^n) \times C^1([0, T], \bbS^n_{++})
		\\
		(\bar x_0, \Sigma_0)                    & \mapsto (\bar x, \Sigma),
	\end{align*}
	where $(\bar x, \Sigma)$ is the Nash equilibrium generated by $\gamma$ using initial condition $(\bar x_0, \Sigma_0)$. Let us define the \textbf{control mapping}
	\begin{align*}
		\varphi_\gamma: \RR^n \times \bbS^{n}_{++} & \to C^1([0, T], \RR^n) \times C^1([0, T], \RR^{n\times n})
		\\
		(\bar x_0, \Sigma_0)                      & \mapsto (R^{-1}B^\dag r, R^{-1}B^\dag P).
	\end{align*}
\end{definition}

\subsection{Objectives}

In this work, we focus on inverse problems concerning the observation of Nash equilibrium mean-covariance trajectories. Assuming an observer has access to several such trajectories starting with different initializations, we investigate the a priori information and the specific number of trajectory observations required to address the following queries:
\begin{enumerate}
	[label=\textbf{(Q\arabic*)},
		ref=\textbf{(Q\arabic*)},
		noitemsep,topsep=0pt]
	\item\label{theory-q1} \textit{Is it possible to recover the optimal control for the observed initializations?}
	\item\label{theory-q2} \textit{Is it possible to recover the optimal control for any arbitrary initialization?}
	\item\label{theory-q3} \textit{Is it possible to recover an admissible tuple of cost parameters $\gamma_2 \in \Gamma_2$ that reproduces the observed mean-covariance trajectories?}
\end{enumerate}

These three questions are distinct yet interrelated. \ref{theory-q1} is conceptually the most straightforward; here, the observer reconstructs the optimal control evaluated exactly along the observed evolution of the population's mean and covariance. In contrast, \ref{theory-q2} requires the observer to reconstruct the optimal control mapping comprehensively, allowing evaluation at new initializations where the macroscopic trajectory is not explicitly given. For initializations distant from the observed data, the observer risks lacking sufficient local information to infer the underlying structure. While \ref{theory-q3} pursues a deeper structural recovery than \ref{theory-q2}, both overarching questions pose significant analytical challenges regarding parameter identifiability.

We formalize the inverse problems with the following definitions.

\begin{definition}
	Let $\gamma \in \Gamma$ and a set of $k\in \NN^*$ initial conditions $\cI = \{(\bar{x}_0^i, \Sigma_0^i)\}_{i=1}^k$ be given. The \textbf{observation operator} $\psi_{\gamma}(\cI)$ is the mapping that assigns to $\gamma$ the corresponding mean-covariance trajectories:
	\begin{equation}
		\psi_{\gamma}(\cI) := \left( \psi_\gamma(\bar x^i_0, \Sigma^i_0) \right)_{i=1, \dots, k}
	\end{equation}
\end{definition}

\begin{definition}
	\label{definition:identifiable-control}
	Let $\hat{\Gamma} \subset \Gamma$. We say that the optimal control is \textbf{locally identifiable} over $\hat{\Gamma}$ from the observations $\psi_\gamma(\cI)$ if there exists a mapping $F$ such that for every $\gamma \in \hat{\Gamma}$ and $(\bar x_0, \Sigma_0) \in \cI$:
	\begin{equation}
		\label{eq:identifiable-control}
		\varphi_\gamma(\bar x_0, \Sigma_0) = F \big(\gamma_1, \psi_{\gamma}(\cI) \big).
	\end{equation}
	We say that the optimal control is \textbf{globally identifiable} over $\hat{\Gamma}$ from the observations $\psi_\gamma(\cI)$ if \eqref{eq:identifiable-control} is satisfied for every $\gamma \in \hat{\Gamma}$ and every initialization $(\bar x_0, \Sigma_0) \in \RR^n \times \bbS^n_{++}$.
\end{definition}

\begin{definition}
	Let $\eta: \Gamma \to \RR^M$ be a quantity of interest for some $M\in \NN^*$ and $\hat{\Gamma} \subset \Gamma$. We say that $\eta$ is \textbf{identifiable} with respect to the observations $\psi_\gamma(\cI)$ over $\hat{\Gamma}$ if there exists a mapping $F$ such that for every $\gamma \in \hat{\Gamma}$:
	\begin{equation}
		\eta(\gamma) = F \big( \gamma_1, \psi_{\gamma}(\cI) \big)
	\end{equation}
\end{definition}

\begin{definition}
	Let $\gamma_1\in \Gamma_1$ and $\gamma_2, \gamma_2'\in \Gamma_2$. We say that $\gamma_2$ and $\gamma_2'$ are \textbf{consistent with respect to} $\gamma_1$, if and only if $\psi_{(\gamma_1, \gamma_2)} = \psi_{(\gamma_1, \gamma_2')}$.
\end{definition}

Before addressing \ref{theory-q1}-\ref{theory-q3}, we establish the mathematical well-posedness of the forward Nash equilibrium and subsequently analyze the injectivity of the trajectory mapping $\psi$.

\section{Existence, uniqueness and injectivity}
\label{section:existence-inj}

\subsection{Existence and uniqueness of Nash equilibrium}

Because our primary focus is the inverse problem, we do not aim for maximal generality in establishing conditions for the system of equations. The system decomposes into two coupled Riccati-type boundary value problems (\eqref{eq:p_dyn}, \eqref{eq:barx_dyn}--\eqref{eq:r_dyn}) and one Sylvester equation \eqref{eq:sigma_dyn}. We refer to \cite{AbouKandilFreilingIonescuJank2003} for a comprehensive treatment of such equations\arxiv{, and we recall two crucial theorems from this text in Appendix~\ref{appendix:riccati}}.

We now state the existence and uniqueness of the Nash equilibrium.
\begin{theorem}
	There exists a unique solution $P \in C^1([0, T], \bbS_+^n)$ to \eqref{eq:p_dyn} satisfying $P_T = Q_T$.
\end{theorem}
\begin{arxivonly}
\begin{proof}
	We refer to \ricattiA, which requires $BR^{-1} B^\dag \in \bbS^n_+$, $Q \in \bbS^n_+$, and $Q_T \in \bbS^n_+$. These conditions are satisfied since $R \in \bbS^n_{++}$ and $Q \in \bbS^n_+$ by the definition of $\Gamma$.
\end{proof}
\end{arxivonly}

Since \eqref{eq:sigma_dyn} is a linear differential equation in $\Sigma$, the Cauchy-Lipschitz theorem guarantees a unique forward solution for any given $P$.

The coupled system \eqref{eq:barx_dyn}--\eqref{eq:r_dyn} is non-trivial, as it involves a mixed initial condition for $\bar x$ and a terminal condition for $r$. This forms a linear boundary value problem associated with a Riccati differential equation.

\begin{theorem}
	\label{theorem:existence_uniqueness_mfg_lqg}
	For any $\gamma \in \Gamma$, there exists a unique solution to the coupled system \eqref{eq:barx_dyn}--\eqref{eq:r_dyn}.
\end{theorem}
\begin{cdconly}
\begin{proof}[Sketch]
	\ricattiB~establishes that \eqref{eq:barx_dyn}--\eqref{eq:r_dyn} admits a unique solution as soon as the Riccati equation
	\begin{equation}
		\label{eq:Z}
		\begin{split}
			\dot Z_t = & - BR^{-1}B^\dag + (A - BR^{-1}B^\dag P_t) Z_t \\
			&+ Z_t(A^\dag - P_t B R^{-1} B^\dag ) + \frac{1}{2} Z_t \bar Q Z_t
		\end{split}
	\end{equation}
	with $Z_0 = 0$ admits a solution on $[0, T]$ and $\det(\frac{1}{2}\bar Q_T Z_T - I_n) \ne 0$. Equation~\eqref{eq:Z} is a symmetric Riccati equation, which we solve by applying \ricattiA. The same theorem yields $Z_T \in \bbS^n_-$. 
	Since $\bar Q_T\in\bbS^n_+$, the spectrum of $\bar Q_TZ_T$ is non-positive, so the determinant condition follows.
\end{proof}
\end{cdconly}
\begin{arxivonly}
\begin{proof}
	We rely on \ricattiB in the appendix, which states that if the Riccati equation
	\begin{equation}
		\label{eq:Z}
		\begin{split}
			\dot Z_t = & - BR^{-1}B^\dag + (A - BR^{-1}B^\dag P_t) Z_t \\
			&+ Z_t(A^\dag - P_t B R^{-1} B^\dag ) + \frac{1}{2} Z_t \bar Q Z_t
		\end{split}
	\end{equation}
	with $Z_0 = 0$ admits a solution on $[0, T]$, and if $\det(\frac{1}{2}\bar Q_T Z_T - I_n) \ne 0$, then~\eqref{eq:barx_dyn} and \eqref{eq:r_dyn} admit a unique solution. To show the existence of a solution to~\eqref{eq:Z}, we apply \ricattiA. One subtlety is that \ricattiA is formulated with a terminal condition. Let us perform a time reversal in~\eqref{eq:Z} by defining $V_t = Z_{T-t}$, which satisfies:
	\begin{equation}
		\label{eq:Z_inv}
		\dot V_t = BR^{-1}B^\dag - U_t V_t  -V_tU_t^\dag - \frac{1}{2}V_t \bar Q V_t,
	\end{equation}
	with $V_T = 0$, and $U_t = (A - BR^{-1}B^\dag P_{T-t})$. We then apply \ricattiA to the dynamics of $-V$, noting that $BR^{-1}B^\dag \in \bbS^n_+$ and $\bar Q \in \bbS^n_+$.

	The theorem further implies $-V_t \in \bbS^n_+$, which implies $V_t \in \bbS^n_-$. Defining $Z_t = V_{T-t}$ for $t \in [0, T]$, $Z$ satisfies Equation~\eqref{eq:Z}. Finally, we verify $\det(\frac{1}{2}\bar Q_T Z_T - I_n) \ne 0$. This holds because $\bar Q_T \in \bbS^n_+$ and $Z_T \in \bbS^n_-$, implying $\bar Q_T Z_T$ has non-positive eigenvalues, which ensures that $(\frac{1}{2}\bar Q_T Z_T - I_n)$ has eigenvalues less than or equal to $-1$, and thus is invertible.
\end{proof}
\end{arxivonly}

We conclude with a lemma bridging the initial and terminal states.
Intuitively, the terminal position $\bar x_T$ encodes information regarding the terminal cost matrix $Q_T$. The following lemma confirms that observing trajectories spanning a basis of initial states prevents degenerate terminal states.

\begin{lemma}
	\label{lemma:barx0_mapsto_barxT}
	Let $\gamma \in \Gamma$.
	\begin{enumerate}
		\item The mapping $\bar x_0 \mapsto \bar x_T$ is linear, invertible, and independent of the initialization $\Sigma_0$.
		\item In particular, for any basis $(\bar x_0^1, \dots, \bar x_0^n)$ of $\RR^n$ and any $\Sigma_0^1, \dots, \Sigma^n_0 \in \bbS^n_{++}$, the terminal mean values $(\bar x^1_T, \dots, \bar x_T^n)$ form a basis of $\RR^n$.
	\end{enumerate}
\end{lemma}
\begin{arxivonly}
\begin{proof}
	The solution $P$ of Equation~\eqref{eq:p_dyn} is independent of the initialization $(\bar x_0, \Sigma_0)$. Building on the proof of Theorem~\ref{theorem:existence_uniqueness_mfg_lqg}, one can show that:
	$$
		\bar x_t = Z_t r_t + \Phi_t\bar{x}_0
	$$
	where $\Phi_t$ is the unique fundamental matrix satisfying $\dot \Phi_t = (\frac{1}{2}Z_t\bar Q + A - B R^{-1} B^\dag P_t)\Phi_t$ with $\Phi_0 = I_n$.
	At the terminal time $T$, and using $r_T = \frac{1}{2}\bar Q_T \bar x_T$, this yields:
	$$
		(I_n -\frac{1}{2} Z_T \bar Q_T) \bar x_T = \Phi_T\bar{x}_0.
	$$
	Since $\bar Q_T \in \bbS^n_+$ and $Z_T \in \bbS^n_-$, the matrix $(I_n -\frac{1}{2} Z_T \bar Q_T)$ has positive eigenvalues and is invertible. Thus:
	$$
		\bar x_T = (I_n - \frac{1}{2}Z_T \bar Q_T )^{-1}\Phi_T\bar{x}_0.
	$$
	The transition matrix $\Phi_T$ is clearly invertible, making the mapping $\bar x_0 \mapsto \bar x_T$ a composition of invertible linear operators.
\end{proof}
\end{arxivonly}

\subsection{Injectivity of the trajectory mapping}
\label{section:canonical_class}

From the viewpoint of inverse problems, a fundamental question is the following: \emph{Is it possible to find two distinct tuples of parameters $\gamma, \gamma' \in \Gamma$, with $\gamma \ne \gamma'$, such that the generated Nash equilibria are identical?}

We answer this question positively by showing that for any tuple of parameters $\gamma = (A, B, \sigma, R, Q, Q_T, \bar Q, \bar Q_T)$, there exists $R'$ such that $\gamma' = (A, I_n, \sigma, R', Q, Q_T, \bar Q, \bar Q_T)$ yields the same Nash equilibrium trajectories as $\gamma$.
In the remainder of this section, we assume \emph{the dynamics parameters $\gamma_1 \in \Gamma_1$ are fixed and known}.

\subsubsection{Control parameter}\;

The matrix $B$ dictates the control's authority over the state space. When $B$ is invertible, the agent possesses full actuation capability. We establish an exact equivalence between the original game and a normalized problem where $(B,R)$ is replaced by $B' = I_n$ and $R' = (B^\dag)^{-1} R B^{-1}$, preserving the resulting mean-covariance trajectory.

\begin{theorem}
	Let $\gamma = (\gamma_1, \gamma_2) \in \Gamma$ with $\gamma_1 = (A, B, \sigma)$ and $\gamma_2 = (R, Q, Q_T, \bar Q, \bar Q_T)$. Set $R' = (B^\dag)^{-1} R B^{-1}$, $\gamma_1' = (A, I_n, \sigma) \in \Gamma_1$, $\gamma_2' = (R', Q, Q_T, \bar Q, \bar Q_T) \in \Gamma_2$ and $\gamma' = (\gamma_1', \gamma_2') \in \Gamma$. Then $\psi_\gamma = \psi_{\gamma'}$.
\end{theorem}

\begin{proof}
	Since $B R^{-1} B^\dag = I_n R'^{-1} I_n$, the Riccati system \eqref{eq:p_dyn}--\eqref{eq:r_dyn} induced by $\gamma'$ is exactly the same as the one induced by $\gamma$.
\end{proof}

\subsubsection{Cost parameters}\;

Assuming $\gamma_1 = (A, B, \sigma) \in \Gamma_1$ is fixed, any two sets of cost parameters $\gamma_2, \gamma_2' \in \Gamma_2$ generating identical optimal feedback policies $\alpha_t = \alpha'_t$ will trivially produce the same Nash equilibrium. Given the feedback form $\alpha_t = - R^{-1}B^\dag P_t X_t - R^{-1}B^\dag r_t$, the operator $R^{-1}B^\dag$ plays a defining role. For convenience, we define the transformed quantity
$$\tilde \eta = R^{-1}B^\dag \eta$$
for any parameter $\eta \in \{P, r, Q, Q_T, \bar Q, \bar Q_T\}$.

\begin{lemma}
	\label{lem:control-eq}
	Let $\gamma_2 = (R, Q, Q_T, \bar Q, \bar Q_T)\in \Gamma_2$ and $\gamma_2' = (R', Q', Q_T', \bar Q', \bar Q_T') \in \Gamma_2$. If 
	$$
		\tilde P_t = \tilde P'_t, \quad \forall t \in [0,T].
	$$
	and if, for every initial mean $\bar x_0\in\RR^n$, the corresponding solutions satisfy $\tilde r_t=\tilde r_t'$ for every $t\in[0,T]$, then $\gamma_2$ and $\gamma_2'$ are consistent with respect to $\gamma_1$.
\end{lemma}
\begin{arxivonly}
\begin{proof}
	Under these conditions, the trajectories $(\bar x, \Sigma)$ and $(\bar x', \Sigma')$ satisfy identical linear differential equations with matching initial conditions, yielding global equality.
\end{proof}
\end{arxivonly}
We now explore sufficient structural conditions for distinct cost parameters $\gamma_2, \gamma_2'$ to satisfy Lemma~\ref{lem:control-eq}.

\begin{theorem}
	\label{theorem:tildeQ=tildeQ'_cA=cA'_psi=psi'}
	Let $\gamma_2 \in \Gamma_2$. For any $\gamma_2' \in \Gamma_2$,  if
	\begin{align}
		(\tilde{Q}, \tilde{Q}_T, \tilde{\bar Q}, \tilde{\bar Q}_T)
		 & = (\tilde{Q}', \tilde{Q}_T', \tilde{\bar Q}', \tilde{\bar Q}_T'),
		\label{eq:tildeQ=tildeQ'}                                            \\
		R^{-1}B^\dag A^\dag (B^\dag)^{-1}R
		 & = R'^{-1}B^\dag A^\dag (B^\dag)^{-1}R'.
		\label{eq:cA=cA'}
	\end{align}
	are satisfied, then $\gamma_2$ and $\gamma_2'$ are consistent with respect to $\gamma_1$.
\end{theorem}
\begin{proof}
	Assume \eqref{eq:tildeQ=tildeQ'} and \eqref{eq:cA=cA'} hold. Let $\cA = R^{-1} B^\dag A^\dag ( B^\dag)^{-1} R$. The Riccati equations for $\tilde P$ and $\tilde P'$ become:
	\begin{align*}
		\dot{\tilde P}_t  & = -\tilde P_t A - \cA \tilde P_t + \tilde P_t B \tilde P_t - \tilde Q,     \\
		\dot{\tilde P}_t' & = -\tilde P_t' A - \cA \tilde P_t' + \tilde P_t' B \tilde P_t' - \tilde Q.
	\end{align*}
	subject to $\tilde P_T = \tilde P'_T = \tilde Q_T$. By the Cauchy-Lipschitz theorem, the assumed existence of a bounded solution yields $\tilde P \equiv \tilde P'$. The pair $(\bar x, \tilde r)$ thus satisfies
	\begin{equation}
		\label{eq:barx-tilder}
		\begin{cases}
			\dot{\bar x}_t = (A - B\tilde P_t)\bar x_t - B\tilde r_t, \\
			-\dot{\tilde{r}}_t = (\cA - \tilde P_t B)\tilde r_t + \dfrac{1}{2}\tilde{\bar Q}\bar x_t.
		\end{cases}
	\end{equation}
	with $\bar x_0 = \bar x_0$ and $\tilde r_T = \frac{1}{2} \tilde{\bar Q}_T \bar x_T$.
	Note that $(\bar x', \tilde r')$ also solves the last equation. Yet, \ricattiB gives uniqueness of the solution if there exists a solution to the Riccati equation
	\begin{equation}
		\label{eq:tildeZ}
		\dot{\tilde Z}_t = - B  + (A - B \tilde P_t) \tilde Z_t + \tilde Z_t(\cA -  \tilde P_t B) + \frac{1}{2}\tilde Z_t \tilde{\bar Q}\tilde Z_t,
	\end{equation}
	with $\tilde Z_0 = 0$, and that $\det(\frac{1}{2}\tilde{\bar Q}_T\tilde Z_T - I_n)\ne 0$. Let $Z$ be the solution of \eqref{eq:Z}. We proved that $\det (\frac{1}{2}\bar Q_T Z_T - I_n)\ne 0$. That implies that $\det(\frac{1}{2}R^{-1}B^\dag \bar Q_T Z_T  (B^\dag)^{-1} R - I_n)\ne 0$ since $R, B$ are invertible. The last expression can be written $\det(\frac{1}{2}\tilde{\bar Q}_T Z_T  (B^\dag)^{-1} R  - I_n)\ne 0$. Furthermore, $Z (B^\dag)^{-1} R$ verifies \eqref{eq:tildeZ}. We deduce that the boundary value problem \eqref{eq:barx-tilder} admits a unique solution. Therefore, $(\bar x, \tilde r) = (\bar x', \tilde r')$. We conclude with Lemma~\ref{lem:control-eq}.
\end{proof}

\begin{remark}
	We highlight a few situations that satisfy or reduce the assumptions of Theorem~\ref{theorem:tildeQ=tildeQ'_cA=cA'_psi=psi'}:
	\begin{itemize}
		\item The theorem holds specifically for $\gamma_2' = \lambda \gamma_2$ with $\lambda >0$.
		\item When $A = \lambda I_n$, $\lambda \in \RR$, \eqref{eq:cA=cA'} is satisfied immediately.
	\end{itemize}
\end{remark}

\begin{remark}
	Equation~\eqref{eq:tildeQ=tildeQ'} does not encompass all constraints on $\gamma_2'$, as the cost matrices must remain positive semi-definite. Consequently, finding an $R' \in \bbS^n_{++}$ that satisfies \eqref{eq:cA=cA'} and scaling the remaining matrices via $(B^\dag)^{-1} R' \tilde \eta$ does not automatically guarantee $\gamma_2' \in \Gamma_2$.
\end{remark}

\section{Main results}
\label{section:main_results}

We first provide a high-level description of our method, and subsequently address Question~\ref{theory-q1}, followed by Questions~\ref{theory-q2}--\ref{theory-q3}.

\subsection{High-level description}

Let $\gamma \in \Gamma$ and $k \in \NN^*$. Assume an observer has access to exact continuous-time trajectories $\psi_\gamma(\cI)$ originating from initializations $\cI = \{(\bar x_0^i, \Sigma_0^i)\}_{i = 1}^k$. The observer aims to: first, obtain the optimal controls associated with these observations \ref{theory-q1}; second, recover the optimal control for any new initialization that may not have been observed \ref{theory-q2}; and finally, infer cost parameters consistent with the observations \ref{theory-q3}.

To answer \ref{theory-q1}, the observer seeks to recover $\varphi_\gamma(\bar x_0^i, \Sigma_0^i) = (\tilde r^i, \tilde P)$ for $i \in [k]^*$. Let us rewrite \eqref{eq:p_dyn}--\eqref{eq:r_dyn} in terms of the "tilde" quantities:
\begin{align}
	\dot{\tilde P} + \tilde PA + \cA\tilde P - \tilde P B \tilde P + \tilde Q      & = 0, \label{eq:p_dyn_tilde}     \\
	\dot{\Sigma} - (A - B \tilde P)\Sigma - \Sigma(A - B \tilde P)^\dag - 2\nu     & = 0, \label{eq:sigma_dyn_tilde} \\
	\dot{\bar x} - (A - B \tilde P)\bar x + B \tilde r                             & = 0, \label{eq:barx_dyn_tilde}  \\
	\dot{\tilde r} + (\cA - \tilde P B)\tilde r + \frac{1}{2}\tilde{\bar{Q}}\bar x & = 0, \label{eq:r_dyn_tilde}
\end{align}
subject to $\tilde P_T = \tilde Q_T$ and $\tilde r_T = \frac{1}{2}\tilde{\bar Q}_T\bar x_T$.
A key observation is that $\tilde P$ does not depend on the initialization and can be expressed solely as a function of $\gamma$. Therefore, if the observer recovers $\tilde P$ from any $i \in [k]^*$, then "half" of the optimal control is known for every initialization. Once $\tilde P$ is identified, the observer can utilize \eqref{eq:barx_dyn_tilde} to recover $\tilde r^i$ for each $i \in [k]^*$, since $\bar x^i$ is directly accessible.

To address \ref{theory-q2}, the observer seeks $\varphi_\gamma(\bar x_0, \Sigma_0) = (\tilde r, \tilde P)$ for an arbitrary $(\bar x_0, \Sigma_0) \in \RR^n \times \bbS^n_{++}$.
While $\tilde P$ is already recovered from $\psi_\gamma(\cI)$, recovering $\tilde r$ via \eqref{eq:barx_dyn_tilde} requires computing $\bar x$, which is not directly observed for a new initialization. Consequently, the observer must find cost parameters consistent with the observed trajectories, thereby necessitating a resolution to \ref{theory-q3}.

Regarding the recovery of $\tilde P$, utilizing the covariance dynamics~\eqref{eq:sigma_dyn_tilde} is promising, as they are linear in $\tilde P$ and involve no other unknown quantities. However, the symmetry of $\Sigma$ introduces a constraint: a single observed trajectory may only identify $\frac{n(n+1)}{2}$ coefficients of $\tilde P_t$ when $n > 1$. Thus, observing two trajectories from sufficiently diverse initializations is generally required. Definition~\ref{definition:hatGamma1} formalizes the required informativeness of these trajectories.
\subsection{Local identifiability of controls}

We now present the main results regarding \ref{theory-q1}.

\begin{definition}
	\label{definition:hatGamma1}
	Let $k \in \NN^*$ and $\cI = \{(\bar x_0^i, \Sigma_0^i)\}_{i=1}^k$. Let $\hat \Gamma_1(\cI) \subset \Gamma$ be the set of parameters $\gamma$ such that for any $t \in [0, T]$, the matrix
	$$
		\kappa_t = \begin{bmatrix} \Sigma_t^1 \otimes I_n + (I_n \otimes \Sigma_t^1)\cT_n \\ \vdots \\ \Sigma_t^k \otimes I_n + (I_n \otimes \Sigma_t^k)\cT_n \end{bmatrix} \in \RR^{kn^2 \times n^2}
	$$
	is of rank $n^2$. We recall that $\cT_n$ is defined in the notation paragraph of Section~\ref{section:background_pb}.
\end{definition}

\begin{lemma}
	\label{lemma:tildeP_identifiable}
	Let $k \in \NN^*$ such that $\min(2, n) \leq k$, and let $\cI = \{(\bar x_0^i, \Sigma_0^i)\}_{i=1}^k$. For every $t\in[0,T]$, the matrix $\tilde P_t$ is identifiable from the knowledge of $(A, B, \sigma)$ and the observations $\psi_\gamma(\cI)$ over $\hat\Gamma_1(\cI)$. Explicitly,  
	$$
	\VEC(B \tilde P_t) = \arg \min_{L \in \RR^{n^2}} \| \kappa_t L - \kappa_t'\|^2,
	$$
	where $\kappa_t$ is as in Definition~\ref{definition:hatGamma1}, $\Upsilon^i_t = -\dot{\Sigma}^i_t + A\Sigma^i_t + \Sigma^i_t A^\dag + 2\nu$ and $\kappa_t ' = \bigl(\VEC(\Upsilon^1_t)^\dag, \dots, \VEC(\Upsilon^k_t)^\dag\bigr)^\dag \in \RR^{kn^2}.$
\end{lemma}

\begin{arxivonly}
\begin{proof}
	Equation~\eqref{eq:sigma_dyn_tilde} implies for each $t \in [0, T]$:
	\begin{equation*}
		B\tilde P_t \Sigma_t + \Sigma_t \tilde{P}^\dag_t B^\dag = -\dot{\Sigma}_t + A\Sigma_t + \Sigma_t A^\dag + 2\nu.
	\end{equation*}
	The kernel of the mapping $H \mapsto H \Sigma_t + \Sigma_t H^\dag$ is $\{ S\Sigma_t^{-1} \mid S = -S^\dag \}$ and we deduce that its rank is $\frac{n(n+1)}{2}$; thus, a single trajectory is insufficient to identify $\tilde P$ for $n>1$. Vectorizing the equation gives:
	\begin{equation*}
		(\Sigma_t \otimes I_n + (I_n \otimes \Sigma_t)\cT_n) \VEC(B\tilde P_t) = \VEC(\Upsilon_t),
	\end{equation*}
	where $\Upsilon_t = -\dot{\Sigma}_t + A\Sigma_t + \Sigma_t A^\dag + 2\nu$. For the $k$ trajectories $\psi_\gamma(\cI)$, $\tilde P_t$ satisfies:
	\begin{equation*}
		\kappa_t \VEC(B\tilde{P}_t) = \begin{bmatrix} \VEC(\Upsilon_t^1) \\ \vdots \\ \VEC(\Upsilon_t^k) \end{bmatrix},
	\end{equation*}
	where $\kappa$ is defined in Definition~\ref{definition:hatGamma1}.
	Denoting the right-hand term by $\kappa_t'$, we deduce:
	\begin{equation*}
		\VEC(B\tilde P_t) = \arg\min_{L \in \RR^{n^2}} \| \kappa_t L - \kappa'_t \|^2.
	\end{equation*}
	Since $B$ is invertible and $\kappa_t$ is of rank $n^2$ by definition of $\hat \Gamma_1(\cI)$, the identification of $\tilde P_t$ follows.
\end{proof}
\end{arxivonly}

\begin{remark}
	The minimal value of $k$ such that $\hat\Gamma_1(\cI)$ is non-empty is at least $2$ when $n > 1$, and $1$ otherwise.
\end{remark}

\begin{theorem}
	\label{theorem:locally-identifiable}
	Let $k \in \NN^*$ such that $\min(2, n) \leq k$, and let $\cI = \{(\bar x_0^i, \Sigma_0^i)\}_{i=1}^k$. The optimal control is locally identifiable over $\hat\Gamma_1(\cI)$ of Definition~\ref{definition:hatGamma1} with respect to the knowledge $(A, B, \sigma)$ and the observations $\psi_\gamma(\cI)$.
\end{theorem}
\begin{proof}
	According to Lemma~\ref{lemma:tildeP_identifiable}, $\tilde P_t$ is identifiable for every $t\in[0,T]$ from the knowledge of $(A, B, \sigma)$ and the observations $\psi_\gamma(\cI)$ over $\hat\Gamma_1(\cI)$. Rearranging the terms in \eqref{eq:barx_dyn_tilde} for each trajectory $i \in [k]^*$ yields for any $t\in [0, T]$:
	$
		\tilde r^i_t = B^{-1} [ (A  - B \tilde P_t) \bar x^i_t - \dot{\bar x}^i_t ].
	$
	This concludes the proof.
\end{proof}

We now address Questions~\ref{theory-q2}-\ref{theory-q3}.

\subsection{Global identifiability and consistent cost parameters}

Section~\ref{section:canonical_class} demonstrated that cost coefficients cannot be recovered uniquely from the observation of mean-covariance trajectories alone. It suggests, however, that the quantity $\cA = R^{-1} B^\dag A^\dag (B^\dag)^{-1} R$ and the set $\{\tilde Q, \tilde Q_T, \tilde{\bar Q}, \tilde{\bar Q}_T\}$ play a crucial role. We first show that these quantities can be recovered from the observation of trajectories, and subsequently provide a method to obtain consistent values for the cost parameters $\gamma_2$.

To identify the relevant quantities, we require observations starting from sufficiently diverse initializations. To identify $\cA$, we aim to use Equation~\eqref{eq:p_dyn_tilde} by first identifying $\tilde P$. To eliminate the unknown $\tilde Q$, we differentiate the equation. Upon differentiation, $\cA$ is multiplied by $\dot{\tilde P}_t$. If $\dot{\tilde P}_t$ is invertible for at least one $t \in [0, T]$, we can recover $\cA$. However, this invertibility condition is overly restrictive: since $\cA$ is a constant and we have access to the full curvature of $\tilde P$, we may identify $\cA$ by leveraging the entire trajectory of $\tilde P$. Definition~\ref{definition:hatGamma2} specifies the requisite informativeness of the $\tilde P$ trajectory.

\begin{definition}
	\label{definition:hatGamma2}
	Let $\hat \Gamma_2 \subset \Gamma$ be the set of parameters $\gamma$ such that there exist $m \in \NN^*$ and a subdivision $0 \leq t_0 < t_1 < \dots < t_{m-1} \leq T$ for which the matrix $\begin{pmatrix} \dot{\tilde P}_{t_0} & \cdots & \dot{\tilde P}_{t_{m-1}}\end{pmatrix}^\dag $
	is of rank $n$.
\end{definition}

\begin{remark}
	If there exists at least one $t \in [0, T]$ such that $\dot{\tilde P}_t$ is invertible, this condition is satisfied.
\end{remark}

\begin{lemma}
	\label{lemma:(A,B,Sigma)_to_cA}
	Let $k \in \NN^*$ such that $\min(2, n) \leq k$, and let $\cI = \{(\bar x_0^i, \Sigma_0^i)\}_{i=1}^k$. The quantities $\cA, \tilde Q, \tilde Q_T$ are identifiable with respect to the observations $\psi_\gamma(\cI)$ over $\hat\Gamma_1(\cI) \cap \hat\Gamma_2$. Explicitly, 
	$
	\VEC(\cA) = \arg \min_{L\in \RR^{n^2}} \| \xi L - \xi' \|^2,
	$
	where, for a subdivision $t_0 < \dots < t_{m-1}$ as in Definition~\ref{definition:hatGamma2}, $\xi = (\dot{\tilde P}_{t_0} \otimes I_n, \dots, \dot{\tilde P}_{t_{m-1}} \otimes I_n)^\dag$, $\Xi_t = -(\ddot{\tilde P}_t + \dot{\tilde P}_t A - \dot{\tilde P}_t B \tilde P_t - \tilde P_t B \dot{\tilde P}_t)$ and $\xi' = \bigl(\VEC(\Xi_{t_0})^\dag, \dots, \VEC(\Xi_{t_{m-1}})^\dag\bigr)^\dag$. Finally, for any $t \in [0, T]$,
	\begin{equation}
		\label{eq:rec-tildeQ}
		\tilde Q = -(\dot{\tilde P}_t + \tilde P_t A + \cA \tilde P_t - \tilde P_t B \tilde P_t), \quad \tilde Q_T = \tilde P_T.
	\end{equation}
\end{lemma}

\begin{arxivonly}
\begin{proof}
	The quantity $\cA$ appears in the derivative of the transformed Riccati equation:
	\begin{equation*}
		\dot{\tilde P}_t + \tilde P_t A + \cA \tilde P_t - \tilde P_t B\tilde P_t + \tilde Q = 0.
	\end{equation*}
	By differentiating this equation with respect to time, we eliminate the unknown $\tilde Q$:
	\begin{equation*}
		\ddot{\tilde P}_t + \dot{\tilde P}_t A + \cA \dot{\tilde P}_t - \dot{\tilde P}_t B\tilde P_t - \tilde P_t B \dot{\tilde P}_t = 0.
	\end{equation*}
	Rearranging terms yields:
	\begin{equation}
		\cA \dot{\tilde P}_t = -(\ddot{\tilde P}_t + \dot{\tilde P}_t A - \dot{\tilde P}_t B \tilde P_t - \tilde P_t B \dot{\tilde P}_t).
	\end{equation}
	Lemma~\ref{lemma:tildeP_identifiable} states that $\tilde P$ is identifiable with respect to the knowledge of $(A, B, \sigma)$ and the observations $\psi_\gamma(\cI)$ over $\hat\Gamma_1(\cI)$. By definition of $\hat \Gamma_2$, we know that there exist $m \in \NN^*$ and a subdivision $0\leq t_0<\cdots < t_{m-1} \leq T$ as in Definition~\ref{definition:hatGamma2}. Assuming access to $(\tilde P_{t_0}, \dots, \tilde P_{t_{m-1}})$, we can formulate the linear system:
	\begin{equation*}
		\begin{pmatrix}
			\dot{\tilde P}_{t_0}^\dag \otimes I_n \\ \vdots \\ \dot{\tilde P}_{t_{m-1}}^\dag \otimes I_n
		\end{pmatrix} \VEC(\cA) =
		\begin{pmatrix}
			\VEC(\Xi_{t_0}) \\ \vdots \\ \VEC(\Xi_{t_{m-1}})
		\end{pmatrix},
	\end{equation*}
	where $\Xi_t = -(\ddot{\tilde P}_t + \dot{\tilde P}_t A - \dot{\tilde P}_t B \tilde P_t - \tilde P_t B \dot{\tilde P}_t)$. Denoting the left-hand matrix by $\xi$ and the right-hand vector by $\xi'$, we obtain:
	\begin{equation*}
		\VEC(\cA) = \arg\min_{L \in \RR^{n^2}} \| \xi L - \xi'\|^2.
	\end{equation*}
	Thus, the identifiability of $\cA$ follows.

	From Equation~\eqref{eq:p_dyn_tilde}, we have for any $t \in [0, T]$:
	$$
		\tilde Q = -(\dot{\tilde P}_t + \tilde P_tA + \cA\tilde  P_t - \tilde P_t B \tilde P_t),
	$$
	with the terminal condition $\tilde Q_T = \tilde P_T$. Thus the identifiability of $\tilde Q, \tilde Q_T$ follows.
\end{proof}
\end{arxivonly}

\begin{remark}
	In the previous lemma:
	\begin{itemize}
		\item we used the property of $\hat \Gamma_1(\cI)$ to identify $\tilde P$, and the property of $\hat \Gamma_2$ to identify $\cA$,
		\item we only used the covariance trajectories and not the mean trajectories,
	\end{itemize}
\end{remark}

It remains to identify $\tilde{\bar Q}$ and $\tilde{\bar Q}_T$. These coefficients govern the dynamics of $\tilde r$. Since $\tilde r \in \RR^n$, its associated equation is $\RR^n$-valued. To uniquely identify the 
matrices $\tilde{\bar Q},\tilde{\bar Q}_T \in \RR^{n\times n}$, we require observations from at least $n$ linearly independent trajectories.

\begin{lemma}
	\label{lemma:tildebarQQ_T_indetifiable}
	Let $\cI = \{(\bar x_0^i, \Sigma_0^i)\}_{i=1}^n$ be such that $(\bar x^1_0, \dots, \bar x^n_0)$ forms a basis of $\RR^n$.
	The quantities $\tilde{\bar Q}$ and $\tilde{\bar Q}_T$ are identifiable with respect to the knowledge $(A, B, \sigma)$ and the observations $\psi_\gamma(\cI)$ over $\hat\Gamma_1(\cI)\cap\hat\Gamma_2$. Explicitly, with $\tilde r^i_t = B^{-1}[(A - B\tilde P_t)\bar x^i_t - \dot{\bar x}^i_t]$ and $\bar \fX_t = (\bar x^1_t \vert \cdots \vert \bar x^n_t)$, $\tilde \fr_t = (\tilde r^1_t \vert \cdots \vert \tilde r^n_t)$,
	\begin{equation}
		\label{eq:rec-tildebarQ}
		\tilde{\bar Q} = -2\bigl(\dot{\tilde \fr}_0 + (\cA - \tilde P_0 B) \tilde \fr_0\bigr) \bar \fX_0^{-1},
		\quad
		\tilde{\bar Q}_T = 2\tilde \fr_T \bar \fX_T^{-1}.
	\end{equation}
\end{lemma}

\begin{arxivonly}
\begin{proof}
	Let $t\in [0, T]$, we know that $\tilde P_t$ is identifiable. From \eqref{eq:barx_dyn_tilde}, we recover $\tilde r^i_t$ for each $i \in \{1, \dots, n\}$ as:
	$$
		\tilde r^i_t = B^{-1} [ (A - B \tilde P_t) \bar x^i_t - \dot{\bar x}^i_t ].
	$$
	Let $\bar \fX_t = (\bar x^1_t \vert \cdots \vert \bar x^n_t)$ and $\tilde \fr_t = (\tilde r^1_t \vert \cdots \vert \tilde r^n_t)$. From Equation~\eqref{eq:r_dyn_tilde}, the parameter $\tilde{\bar Q}$ must satisfy:
	$$
		\tilde{\bar Q} \bar x^i_t = -2(\dot{\tilde r}^i_t + (\cA - \tilde P_t B) \tilde r^i_t ).
	$$
	Aggregating these over the $n$ trajectories, we obtain:
	$$
		\tilde{\bar Q} \bar \fX_t = -2(\dot{\tilde \fr}_t + (\cA - \tilde P_t B)\tilde \fr_t),
	$$
	By assumption, $\bar \fX_0$ is invertible. We have:
	$$
		\tilde{\bar Q} = -2(\dot{\tilde \fr}_0 + (\cA - \tilde P_0 B) \tilde{\fr}_0 ) \bar \fX^{-1}_0.
	$$
	According to Lemma~\ref{lemma:barx0_mapsto_barxT}, $\bar \fX_T$ is invertible.
	Finally, the terminal condition provides:
	$
		\tilde{\bar Q}_T = 2\tilde\fr_T \bar \fX_T^{-1}.
	$
\end{proof}
\end{arxivonly}
We have thus established how to identify $\cA, \tilde Q, \tilde Q_T, \tilde{\bar Q}, \tilde{\bar Q}_T$. The remaining challenge is whether it is possible to find consistent cost parameters $(R, Q, Q_T, \bar Q, \bar Q_T) \in \Gamma_2$.

\begin{theorem}
	\label{theorem:sdp}
	Let $\gamma = (\gamma_1, \gamma_2) \in \Gamma$. Let
	\(
	K(\mathcal A) = \mathcal A^\dag \otimes I_n - I_n \otimes B^\dag A^\dag (B^\dag)^{-1} ,
	\)
	and
	\(
	K(\tilde \cQ) = (\tilde \cQ B)^\dag\otimes I_n - I_n \otimes (\tilde \cQ B)^\dag
	\)
	for \(\tilde \cQ \in \{\tilde Q, \tilde Q_T, \tilde{\bar Q}, \tilde{\bar Q}_T\}\). Let \(K = I_{n^2} - \cT_n \in \mathbb{R}^{n^2 \times n^2}\). Define
	$$
		\mathcal K = \bigl(K^\dag,\, K(\mathcal A)^\dag,\, K(\tilde Q)^\dag,\, K(\tilde Q_T)^\dag,\, K(\tilde{\bar Q})^\dag,\, K(\tilde{\bar Q}_T)^\dag\bigr)^{\dag},
		\quad
	$$
	with $d = \dim(\ker(\mathcal K))$. Let \((G_1,\dots,G_d)\) be a basis of $\ker(\cK)$.
	Then, for \(\epsilon>0\), the problem
	\[
		\min_{\alpha \in \mathbb{R}^d} \|\alpha\|^2
		\quad \text{s.t.} \quad
		\begin{cases}
			(B^\dag)^{-1} R' \tilde \cQ \succeq 0, \; \tilde \cQ \in \{\tilde Q, \tilde Q_T, \tilde{\bar Q}, \tilde{\bar Q}_T \}, \\
			R' \succeq \epsilon I,                                                                                                \\
			\VEC(R') = \sum_{i=1}^d \alpha_i G_i
		\end{cases}
	\]
	has a unique solution. Moreover,
	$
		(R', (B^\dag)^{-1} R' \tilde Q, (B^\dag)^{-1} R' \tilde Q_T, (B^\dag)^{-1} R' \tilde{\bar Q}, (B^\dag)^{-1} R' \tilde{\bar Q}_T) \in \Gamma_2
	$
	is consistent with $\gamma_2$ with respect to $\gamma_1$.
\end{theorem}

\begin{proof}
	Let us first show that the kernel of $\cK$ has dimension greater than or equal to $1$.

	By definition of the $\Gamma_2$, $R \in \bbS^n_{++}$, therefore $K\VEC(R) = 0$. Notice that $K(\cA)\VEC(R) = 0$ if and only if $B^\dag A^\dag (B^\dag)^{-1} R = R\cA$. Yet, by definition, $\cA = R^{-1}B^\dag A^\dag (B^\dag)^{-1} R.$
	Therefore $K(\cA)\VEC(R) = 0$.

	Let $\tilde \cQ \in \{\tilde Q, \tilde Q_T,\tilde{\bar Q}, \tilde{\bar Q}_T\}$. We have $K(\tilde \cQ)\VEC(R) = 0$
	if and only if
	\[
		R\tilde \cQ B = B^\dag \tilde \cQ^\dag R
		\iff (B^\dag)^{-1}R\tilde \cQ = \tilde \cQ^\dag R B^{-1}.
	\]
	This is satisfied since $Q = (B^\dag)^{-1}R\tilde Q$
	is symmetric (the same holds for $Q_T, \bar Q, \bar Q_T$). Hence $\VEC(R) \in \ker(\cK)$. Now, it is easy to verify that there exists a $\lambda \in \RR_+$ such that $\lambda R$ satisfies the constraints of the semidefinite program (SDP).

	Now let us prove that a solution to the semidefinite problem yields consistent cost parameters. Let $\epsilon>0$, and let $\alpha \in \mathbb{R}^d$ be a solution of the problem. Define $R'$ as in the problem and set $\cQ' = (B^\dag)^{-1} R' \tilde \cQ$ for $\cQ \in \{Q, Q_T, \bar Q, \bar Q_T\}$.
	By the same reasoning as above, it follows that $(R', Q', Q'_T, \bar Q', \bar Q'_T)$ satisfies \eqref{eq:tildeQ=tildeQ'} and \eqref{eq:cA=cA'}. Furthermore, the constraints of the problem impose $R' \in \bbS^n_{++}$, $Q', Q'_T, \bar Q', \bar Q'_T  \in \bbS^n_+$.

	Theorem~\ref{theorem:tildeQ=tildeQ'_cA=cA'_psi=psi'} then implies that $(R', Q', Q'_T, \bar Q', \bar Q'_T)$ is consistent with $\gamma_2$ with respect to $\gamma_1$.

	Finally, we show the existence and uniqueness of the solution to the semidefinite problem. The constraints define a non-empty closed convex set, and the objective function $\|\cdot\|^2$ is strictly convex. Since the objective is coercive, it attains a minimum; strict convexity yields uniqueness.
\end{proof}

\begin{corollary}
	Let $\cI = \{(\bar x_0^i, \Sigma_0^i)\}_{i=1}^n$ be such that $(\bar x^1_0, \dots, \bar x^n_0)$ forms a basis of $\RR^n$. The optimal control is globally identifiable over $\hat \Gamma_1(\cI)\cap \hat \Gamma_2$ with respect to the observations $\psi_\gamma(\cI)$.
\end{corollary}

\begin{arxivonly}
\begin{proof}
	Lemmas~\ref{lemma:(A,B,Sigma)_to_cA} and \ref{lemma:tildebarQQ_T_indetifiable} show that $(\cA, \tilde Q, \tilde Q_T, \tilde{\bar Q}, \tilde{\bar Q}_T)$ are identifiable. Theorem~\ref{theorem:sdp} states that we can recover consistent cost parameters $\gamma_2'$ such that $\gamma_2'$ verifies \eqref{eq:tildeQ=tildeQ'} and \eqref{eq:cA=cA'}. In the proof of Theorem~\ref{theorem:tildeQ=tildeQ'_cA=cA'_psi=psi'}, we established that $\varphi_\gamma = \varphi_{(\gamma_1, \gamma_2')}$, which finishes the proof.
\end{proof}
\end{arxivonly}

\begin{arxivonly}
\begin{table}[ht]
	\centering
	\renewcommand{\arraystretch}{1.3} 
	\caption{Identifiability of Cost Parameters. 
    }
	\label{tab:recovery_summary}
	\begin{tabular}{|l|l|c|l|}
		\hline
		\textbf{Known} & \textbf{Observed}       & \textbf{Assumption}                                         & \textbf{Result}                 \\ \hline
		$A, B, \sigma$ & $\psi_\gamma(\cI)$ & $\hat \Gamma_1(\cI)$                                        & \ref{theory-q1}                 \\ \hline
		$A, B, \sigma$ & $\psi_\gamma(\cI)$      & $\hat \Gamma_1(\cI)\cap \hat \Gamma_2$, $\det(\bar \fX_0)\ne0$ & \ref{theory-q2}-\ref{theory-q3} \\ \hline
	\end{tabular}
\end{table}
\end{arxivonly}

\begin{arxivonly}
\section{Numerical experiments}
\label{section:numerical}

In this section, we numerically validate the proposed parameter recovery framework. We consider a two-dimensional state space ($n=2$) with two distinct initial configurations: $\bar{x}_0^1 = (1,0)^\dag$, $\bar{x}_0^2 = (0,1)^\dag$, and initial covariances $\Sigma_0^1 = I_2$ and $\Sigma_0^2 = \begin{pmatrix} 1 & 1/2 \\ 1/2 & 1 \end{pmatrix}$. The control and noise matrices are fixed as $B = I_2$ and $\sigma = 0.2 I_2$.

To evaluate the robustness of the solver, we performed $100$ independent trials. In each trial, the ground-truth parameters $\left( A, R, Q, Q_T, \bar{Q}, \bar{Q}_T \right)$ were randomly sampled from a distribution of symmetric positive definite matrices. The forward equilibrium trajectories $\psi_{\gamma}(\bar x_0, \Sigma_0)$ were generated by solving the coupled system \eqref{eq:p_dyn}-\eqref{eq:r_dyn} using a combination of implicit schemes and Newton's method.

We subsequently applied the inverse mapping to recover the intermediate quantities $\cA_{\text{rec.}}$ and $\tilde \cQ_{\text{rec.}}$ for $\cQ \in \{Q, Q_T, \bar Q, \bar Q_T\}$, providing the necessary inputs for the SDP defined in Theorem~\ref{theorem:sdp} (with $\epsilon = 1$). Due to numerical sensitivities associated with the second-order derivative $\ddot{\tilde P}_t$, we opted to recover $\cA$ and $\tilde{Q}$ by performing a Mean Squared Error (MSE) minimization over the full trajectories using \eqref{eq:p_dyn_tilde} and \eqref{eq:barx_dyn_tilde} as constraints, rather than using the direct algebraic recovery.

The statistical distribution of the relative errors, $\|\cdot_{\text{rec}} - \cdot_{\text{obs}}\|_2 / \|\cdot_{\text{obs}}\|_2$, is reported in Figure~\ref{fig:error_boxplot}. The method achieves high precision in reconstructing the mean-covariance trajectories, with median relative errors typically below $10^{-2}$. Notably, the error in the recovered $R$ matrix remains relatively large. This confirms that the recovered parameters, although not identical to the ground truth, produce an equilibrium nearly indistinguishable from the original. The observed sensitivity in the algebraic recovery of $\cA$ is primarily attributed to the inherent numerical instability of matrix inversion and the noise amplification caused by finite-difference approximations of $\dot \Sigma$.

This is illustrated in Figure~\ref{fig:rec}, which superimposes the original mean-covariance trajectory, generated by the ground-truth parameters, and the recovered one, generated by the parameters returned by our method. The two trajectories are indistinguishable, confirming that the recovered parameters reproduce the observed equilibrium even though they differ from the ground truth.

\begin{figure}[htbp]
	\centering
	\includegraphics[width=0.9\linewidth]{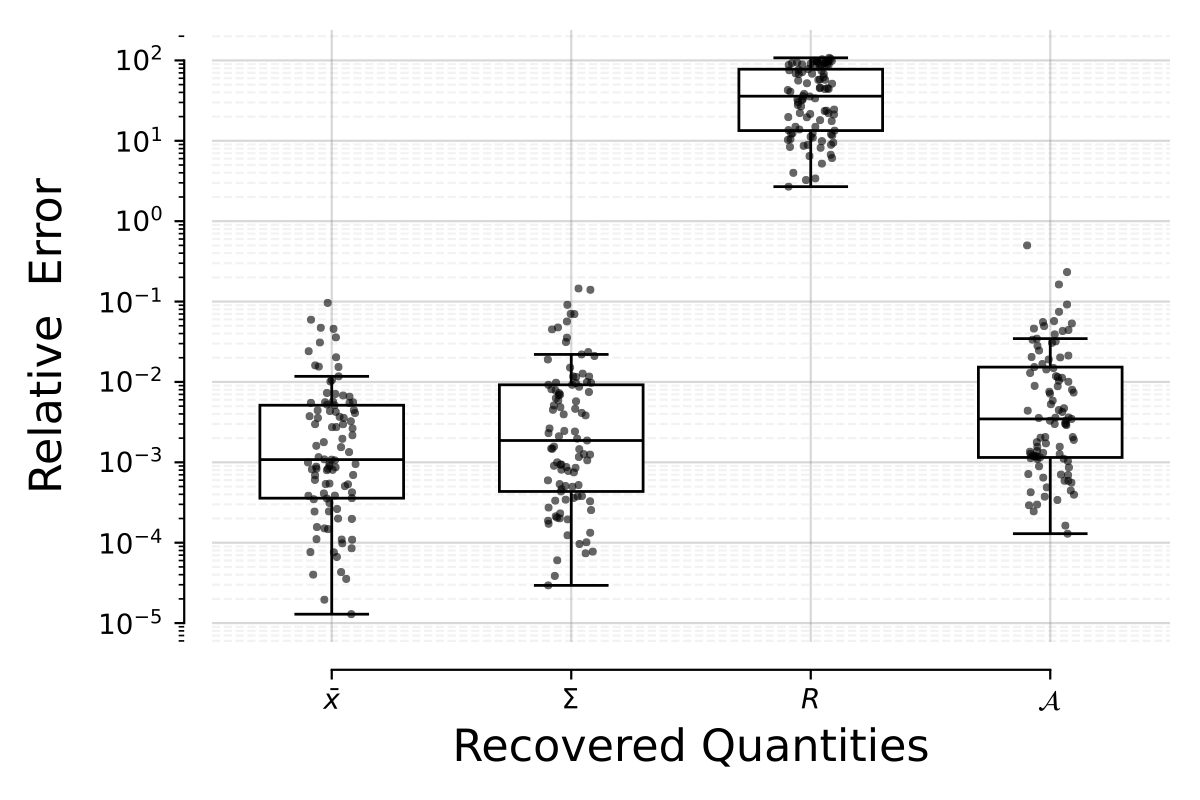}
	\caption{Numerical consistency.}
	\label{fig:error_boxplot}
\end{figure}
\vspace{-2em}
\begin{figure}[htbp]
	\centering
	\includegraphics[width=0.9\linewidth]{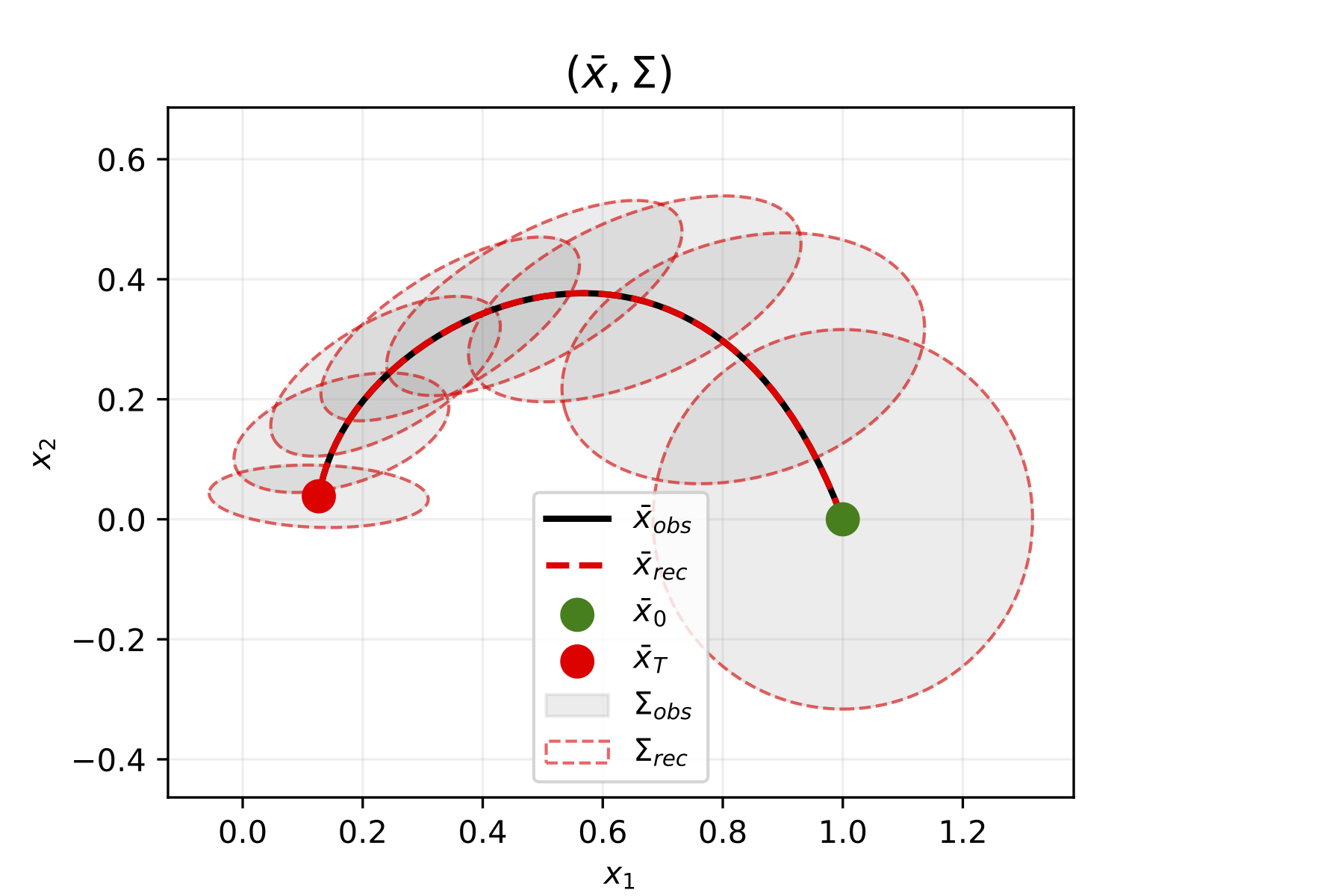}
	\caption{Trajectory comparison: observed vs recovered.}
	\label{fig:rec}
\end{figure}
\end{arxivonly}

\section{Conclusion}
\label{section:conclusion}

\cdc{We established the existence and uniqueness of the Nash equilibrium for a Linear-Quadratic-Gaussian (LQG) Mean Field Game (MFG) framework, and showed that the mapping from admissible parameters to equilibrium mean-covariance trajectories is not injective. We then proved that the optimal control is locally identifiable from sufficiently diverse observed initializations, that generalizing it to unobserved initializations requires recovering consistent cost parameters, and that such parameters follow from the identifiable quantities $\cA, \tilde Q, \tilde Q_T, \tilde{\bar Q}, \tilde{\bar Q}_T$ together with a semidefinite program. A direction for future work is to extend the proposed framework to noisy and partially observed data.}
\arxiv{In this paper, we established the existence and uniqueness of the Nash equilibrium for the Linear-Quadratic-Gaussian (LQG) Mean Field Game (MFG) framework. We demonstrated that the mapping from admissible parameters to equilibrium mean-covariance trajectories is not injective, implying that distinct parameter sets can lead to identical equilibria. Our analysis addressed three key objectives:}

\arxiv{First, we showed that by observing mean-covariance trajectories from sufficiently diverse initializations, it is possible to identify the optimal control for those specific initializations. Second, we demonstrated that generalizing this control to unobserved initializations is not straightforward. We showed that this generalization is achievable by addressing our third objective: the recovery of consistent cost parameters. We proved that these parameters can be identified by first extracting crucial identifiable quantities and subsequently solving a semidefinite program (SDP).}

\arxiv{Our numerical experiments support the theoretical findings. In particular, they show that the proposed method accurately reconstructs mean–covariance trajectories and associated controls, even though the recovered cost parameters may differ significantly from the ground truth. This confirms that consistency with observed equilibria, rather than exact parameter recovery, is the appropriate objective in this setting. Simultaneously, the experiments reveal important numerical challenges, notably the sensitivity of certain identification steps to derivative estimation and matrix inversion.}

\arxiv{These observations point to a clear direction for future work. A key challenge is to extend the proposed framework to noisy and partially observed data, where trajectories are corrupted by estimation errors or sampling effects. It is crucial to address this setting, as it reflects practical real-world scenarios.}



\bibliographystyle{IEEEtran} 
\bibliography{references}

@book{AbouKandilFreilingIonescuJank2003,
  author    = {Abou-Kandil, H. and Freiling, G. and Ionescu, V. and Jank, G.},
  title     = {Matrix Riccati Equations in Control and Systems Theory},
  series = {Systems \& Control: Foundations \& Applications},
  publisher = {Birkh{\"a}user},
  year      = {2003}
}

@book{BensoussanFrehseYam2013,
  author    = {Bensoussan, A. and Frehse, J. and Yam, P.},
  title     = {Mean Field Games and Mean Field Type Control Theory},
  publisher = {Springer},
  year      = {2013}
}

@article{BensoussanSungYamYung2016,
  author  = {Bensoussan, A. and Sung, K. C. J. and Yam, S. C. P. and Yung, S. P.},
  title   = {Linear-Quadratic Mean Field Games},
  journal = {Journal of Optimization Theory and Applications},
  volume  = {169},
  pages   = {496--529},
  year    = {2016}
}

@book{CarmonaDelarue2018,
  author    = {Carmona, R. and Delarue, F.},
  title     = {Probabilistic Theory of Mean Field Games with Applications I: Mean Field FBSDEs, Control, and Games},
  publisher = {Springer},
  year      = {2018}
}

@inproceedings{ChenZiebartAISTATS2015,
  author    = {Chen, X. and Ziebart, B.},
  title     = {Predictive inverse optimal control for linear-quadratic-{G}aussian systems},
  booktitle = {Proceedings of the 18th International Conference on Artificial Intelligence and Statistics (AISTATS)},
  series    = {PMLR},
  volume    = {38},
  pages     = {165--173},
  year      = {2015}
}

@article{bardi2012explicit,
  title={Explicit solutions of some linear-quadratic mean field games},
  author={Bardi, Martino},
  journal={Networks and Heterogeneous Media},
  volume={7},
  number={2},
  pages={243--261},
  year={2012}
}

@article{Graber2016,
  author  = {Graber, P. J.},
  title   = {Linear Quadratic Mean Field Type Control and Mean Field Games with Common Noise, with Application to Production of an Exhaustible Resource},
  journal = {Applied Mathematics \& Optimization},
  volume  = {74},
  number  = {3},
  pages   = {459--486},
  year    = {2016}
}

@inproceedings{JeanMaslovskayaCDC2018,
  author    = {Jean, F. and Maslovskaya, S.},
  title     = {Inverse optimal control problem: the linear-quadratic case},
  booktitle = {Proceedings of the 57th IEEE Conference on Decision and Control (CDC)},
  year      = {2018},
  pages     = {888--893},
  publisher = {IEEE},
  doi       = {10.1109/CDC.2018.8619204}
}

@article{LasryLions2007,
  author    = {Lasry, Jean-Michel and Lions, Pierre-Louis},
  title     = {{Mean field games}},
  journal   = {Japanese Journal of Mathematics},
  volume    = {2},
  number    = {1},
  pages     = {229--260},
  year      = {2007},
  publisher = {Springer}
}

@article{HuangMalhameCaines2006,
  author    = {Huang, Minyi and Malham{\'e}, Roland P. and Caines, Peter E.},
  title     = {Large population stochastic dynamic games: closed-loop {M}c{K}ean-{V}lasov systems and the {N}ash certainty equivalence principle},
  journal   = {Communications in Information and Systems},
  volume    = {6},
  number    = {3},
  pages     = {221--252},
  year      = {2006},
  publisher = {International Press of Boston}
}

@book{AchdouCardaliaguet2020,
  author    = {Achdou, Yves and Cardaliaguet, Pierre and Delarue, Fran{\c{c}}ois and Porretta, Alessio and Santambrogio, Filippo},
  title     = {Mean Field Games: {C}etraro, {I}taly 2019},
  series    = {Lecture Notes in Mathematics},
  volume    = {2281},
  publisher = {Springer International Publishing},
  address   = {Cham},
  year      = {2020}
}

@article{AchdouCapuzzoDolcetta2010,
  author    = {Achdou, Yves and Capuzzo-Dolcetta, Italo},
  title     = {Mean Field Games: Numerical Methods},
  journal   = {SIAM Journal on Numerical Analysis},
  volume    = {48},
  number    = {3},
  pages     = {1136--1162},
  year      = {2010},
  publisher = {Society for Industrial and Applied Mathematics}
}

@article{CarmonaLauriere2021,
  author    = {Carmona, Ren{\'e} and Lauri{\`e}re, Mathieu},
  title     = {Convergence Analysis of Machine Learning Algorithms for the Numerical Solution of Mean Field Control and Games {I}: {T}he Ergodic Case},
  journal   = {SIAM Journal on Numerical Analysis},
  volume    = {59},
  number    = {3},
  pages     = {1455--1485},
  year      = {2021},
  publisher = {Society for Industrial and Applied Mathematics}
}

@article{LiuLoZhang2025,
  author        = {Liu, Hongyu and Lo, Catharine W. K. and Zhang, Shen},
  title         = {Inverse Problems for Mean Field Games},
  year          = {2025},
  journal        = {arXiv 2503.14914}
}

@article{liu2023inverse,
  title={Inverse problems for mean field games},
  author={Liu, Hongyu and Mou, Chenchen and Zhang, Shen},
  journal={Inverse Problems},
  volume={39},
  number={8},
  pages={085003},
  year={2023},
  publisher={IOP Publishing}
}

@article{ZhangYangMouZhou2025,
  author    = {Zhang, Jingguo and Yang, Xianjin and Mou, Chenchen and Zhou, Chao},
  title     = {Learning surrogate potential mean field games via {G}aussian processes: {A} data-driven approach to ill-posed inverse problems},
  journal   = {Journal of Computational Physics},
  volume    = {543},
  pages     = {114412},
  year      = {2025},
  publisher = {Elsevier}
}

@phdthesis{Corrigan1973,
  author = {Corrigan, John D.},
  title  = {Parameter identification applied to linear quadratic differential games},
  year   = {1973},
  type   = {Doctoral Dissertations},
  school = {University of Missouri--Rolla}
}

@inproceedings{chen2024inverse,
  author    = {Zhixing Chen and Lei Guo},
  title     = {An Inverse Problem for Adaptive Linear Quadratic Stochastic Differential Games},
  booktitle = {Proceedings of the 63rd IEEE Conference on Decision and Control (CDC)},
  year      = {2024},
  pages = {1838--1843},
  doi   = {10.1109/CDC56724.2024.10886605},
  publisher = {IEEE}
}

@article{inga2019solution,
  author  = {Jairo Inga and Esther Bischoff and Timothy L. Molloy and Michael Flad and S{\"o}ren Hohmann},
  title   = {Solution Sets for Inverse Non-Cooperative Linear-Quadratic Differential Games},
  journal = {IEEE Control Systems Letters},
  volume  = {3},
  number  = {4},
  pages = {871--876},
  year    = {2019},
  publisher = {IEEE}
}

@article{arxiv-version,
  author  = {Lambrecht, Gr{\'e}goire and Lauri{\`e}re, Mathieu},
  title   = {Inverse Problems for Costs and Controls in {LQG} {MFGs} via Mean Field Trajectories},
  journal = {arXiv:2606.09302},
  year    = {2026}
}

\begin{arxivonly}
\appendix

\section{Riccati Equations}
\label{appendix:riccati}

We recall two useful theorems from \cite{AbouKandilFreilingIonescuJank2003}.

\begin{theorem}[4.1.6., p.186, \cite{AbouKandilFreilingIonescuJank2003}]
	\label{theorem:riccati_1}
	Let $A_t \in \RR^{n\times n}$ for $t \le T$. Let $F\in\bbS_+^n$. If $S_t, Q_t \in \bbS^n_+$ for $t \le T$, then the (unique) solution $\fX$ of the Riccati differential equation
	\begin{equation*}
		\dot{\fX}_t = -A^\dag_t \fX_t - \fX_tA_t - Q_t + \fX_tS_t\fX_t, \quad \fX_T = F
	\end{equation*}
	with piecewise continuous and locally bounded coefficients exists for $t \le T$ with
	$0 \preccurlyeq \fX_t$ for $t \le T;$
\end{theorem}

\begin{theorem}[3.1.3., p.96, \cite{AbouKandilFreilingIonescuJank2003}]
	\label{theorem:riccati_2}
	Let $M(t) = \begin{pmatrix} M_{11} & M_{12} \\ M_{21} & M_{22} \end{pmatrix}(t) \in \mathbb{R}^{2n \times 2n} $ and $K_T \in \bbS^n_+$. If the Riccati differential equation
	\begin{equation*}
		\dot{W} = M_{12} + M_{11}W - WM_{22} - WM_{21}W, \quad
	\end{equation*}
	with $W_0 = 0,\, W_t \in \mathbb{R}^{n \times n} $, admits a solution in $[0, T]$ and if $\det(K_T W_T - I_n) \neq 0$,
	then the boundary value problem
	\begin{align*}
		\frac{d}{dt} \begin{pmatrix} x \\ r \end{pmatrix} & = \begin{pmatrix} M_{11} & M_{12} \\ M_{21} & M_{22} \end{pmatrix} \begin{pmatrix} x \\ r \end{pmatrix}, \quad
	\end{align*}
	with $x_{t= 0} = x_0, \quad r_T = K_T x_T,  $
	has a unique solution.
\end{theorem}
\end{arxivonly}

\end{document}